\documentclass[12pt]{amsart}

\usepackage[usenames]{color}
\usepackage{fullpage,url,mathrsfs,stmaryrd}

\newcommand{\Id}{{\mathop{\operatorname{\rm Id}}}}
\newcommand{\AN}{{\rm{AN}}}
\newcommand{\ana}{{\rm{ana}}}
\newcommand{\eq}{{\rm{eq}}}
\newcommand{\adm}{{\rm{adm}}}

\newcommand{\wadm}{{\rm{wadm}}}

\usepackage{amssymb}
\usepackage[all]{xy}
\usepackage{mathrsfs}
\usepackage{enumerate}
\usepackage{amscd}

\usepackage{bbm}

\DeclareMathOperator{\Ab}{{Ab}}

\DeclareMathOperator{\fin}{{fin}}

\DeclareMathOperator{\good}{{good}}

\DeclareMathOperator{\Ob}{{Ob}}
\DeclareMathOperator{\Pol}{{Pol}}
\DeclareMathOperator{\DGRings}{{DGRings}}
\DeclareMathOperator{\Groups}{{Groups}}
\DeclareMathOperator{\SGroups}{{SGroups}}
\DeclareMathOperator{\Rings}{{Rings}}
\DeclareMathOperator{\SRings}{{SRings}}
\DeclareMathOperator{\Grpds}{{Grpds}}
\DeclareMathOperator{\Cats}{{Cats}}
\DeclareMathOperator{\RGrpds}{{RGrpds}}
\DeclareMathOperator{\AbGrpds}{{Ab-Grpds}}

\newcommand{\MMor}{\underline{\on{Mor}}}

\newcommand{\cC}{{\mathcal C}}

\newcommand{\cN}{{\mathcal N}}

\newcommand{\cQ}{{\mathcal Q}}

\newcommand{\sP}{{\mathscr P}}

\newcommand{\sR}{{\mathscr R}}

\newcommand{\fC}{{\mathfrak C}}

\newcommand{\nc}{\newcommand}

\nc\wh{\widehat}

\nc\on{\operatorname}

\nc\Gr{\on{Gr}}

\nc\Fl{\on{Fl}}

\newtheorem{lem}[subsubsection]{Lemma}
\newtheorem{prop}[subsubsection]{Proposition}

\newtheorem{thm}[subsubsection]{Theorem}

\theoremstyle{remark}

\newcommand{\BZ}{{\mathbb{Z}}}

\DeclareMathOperator{\Cone}{{Cone}}

\newcommand{\limto}{{\displaystyle\lim_{\longrightarrow}}}
\newcommand{\rightlim}{\mathop{\limto}}

\newcommand{\leftlim}{\mathop{\displaystyle\lim_{\longleftarrow}}}
\newcommand{\limfromn}{\leftlim\limits_{\raise3pt\hbox{$n$}}}
\newcommand{\limton}{\rightlim\limits_{\raise3pt\hbox{$n$}}}

\newcommand{\rightlimit}[1]{\mathop{\lim\limits_{\longrightarrow}}\limits%
                    _{\raise3pt\hbox{$\scriptstyle #1$}}}

\newcommand{\leftlimit}[1]{\mathop{\lim\limits_{\longleftarrow}}\limits%
                    _{\raise3pt\hbox{$\scriptstyle #1$}}}

\newcommand{\epi}{\twoheadrightarrow}
\newcommand{\iso}{\buildrel{\sim}\over{\longrightarrow}}
\newcommand{\mono}{\hookrightarrow}

\DeclareMathOperator{\Adm}{{Adm}}

\DeclareMathOperator{\Coker}{{Coker}}
\DeclareMathOperator{\Corr}{{Corr}}

\DeclareMathOperator{\Graph}{{Graph}} 
\DeclareMathOperator{\Hom}{{Hom}}

\DeclareMathOperator{\Ker}{{Ker}} \DeclareMathOperator{\id}{{id}}
\DeclareMathOperator{\im}{{Im}} 

\DeclareMathOperator{\Mor}{{Mor}}
 \DeclareMathOperator{\op}{{op}}

\DeclareMathOperator{\Sets}{{Sets}}
\DeclareMathOperator{\SSets}{{SSets}}

\theoremstyle{definition}

\numberwithin{equation}{section}

\newcommand{\Funct}{\operatorname{Funct}}

\begin{document}
\title[Ring groupoids]{On a notion of ring groupoid}
\author{Vladimir Drinfeld}
\address{University of Chicago, Department of Mathematics, Chicago, IL 60637}
\email{drinfeld@math.uchicago.edu}

\begin{abstract}
By a ring groupoid we mean an animated ring whose $i$-th homotopy groups are zero for all $i>1$.

In this expository note we give an elementary treatment of the $(2,1)$-category of ring groupoids (i.e., without referring to general animated rings and without using $n$-categories for $n>2$).
The note is motivated by the fact that ring stacks play a central role in the Bhatt-Lurie approach to prismatic cohomology.
\end{abstract}

\maketitle

\section{Introduction}

\subsection{Subject of this note}
This note is expository. Following an idea of Lawvere \cite{La1, La2}, we introduce a notion of ring groupoid\footnote{We do not claim that it is the only reasonable notion, see \S\ref{sss:Picard groupoids}.} and a slightly more general notion of ring object in a (2,1)-category.
Then we recall an elementary description of the $(2,1)$-category of ring groupoid; the idea (which goes back to B.~Noohi~ \cite{N0}) is to use the magic word ``extension''.

The (2,1)-category of ring groupoids is a full subcategory of the $\infty$-category of \emph{animated rings}, see \S\ref{sss:animated rings}.
But we do not emphasize this point of view. On the contrary, our exposition of the notion of ring groupoid goal is \emph{elementary} (we do not use $n$-categories for $n>2$).

\subsection{Motivation}
The notion of ring stack\footnote{By a ring stack on a site $S$ we mean a ring object in the $(2,1)$-category of stacks on $S$. Equivalently, it is a prestack of ring groupoids which happens to be a stack.} plays a central role in the Bhatt-Lurie approach to prismatic cohomology, see \cite[\S 1.3-1.4]{Dr}. 

\subsection{Organization}
In \S\ref{s:the 2-category} we define the $(2,1)$-category of ring groupoids. In \S\ref{s:the 1-category} we define and describe the naive $1$-category of ring groupoids. 
In \S\ref{s:RGrpds_AN} we use the description of the  1-category to describe the $(2,1)$-category; the main result (Theorem~\ref{t:main}) goes back to \cite{N0,AN1}.
In \S\ref{s:anafunctors} we recall the notion of \emph{anafunctor} from M.~Makkai's work \cite{Mak} (this notion is closely related to the material from \S\ref{s:RGrpds_AN}).

Let us note that \S\ref{s:anafunctors} and the related \S\ref{ss:admissibility for categories} can be read independently of the rest of the article.

\subsection{Acknowledgements}
I thank A.~ Mathew and N.~Rozenblyum for useful discussions. In particular, they recommended me to define the $(2,1)$-category of ring groupoids using Lawvere's approach. Moreover, \S\ref{sss:Q & DGRings} and \S\ref{ss:Eilenberg-Zilber} are due to A.~ Mathew.

The author's work was partially supported by NSF grant DMS-2001425.

\section{Definition of the $(2,1)$-category of ring groupoids}   \label{s:the 2-category}
\subsection{Lawvere's observation}   \label{ss:Lawvere}
\subsubsection{Convention}
All rings are assumed to be commutative, associative, and unital (unless said otherwise).

\subsubsection{Notation}   \label{sss:notation}
Let $\Rings$ be the category of all rings. Let $\Pol\subset\Rings$ be the full subcategory of free rings. Let $\Pol_{\fin}\subset\Pol$ be  the full subcategory of finitely generated free rings (i.e., rings isomorphic to $\BZ [x_1,\ldots ,x_n]$ for some $n\ge 0$).

\subsubsection{Lawvere's observation}  \label{sss:Lawvere}
Consider the functors
\begin{equation}   \label{e:first functor}
\Rings\to\Funct_{\Pi}(\Rings^{\op} ,\Sets)\to \Funct_{\Pi}(\Pol^{\op} ,\Sets)
\end{equation}
where $\Funct_{\Pi}$ stands for the category of those functors that commute with products and the first arrow in \eqref{e:first functor} is the Yoneda embedding. We also have a canonical functor
\begin{equation}  \label{e:second functor}
\Funct_{\Pi}(\Pol^{\op} ,\Sets)\to\Funct_{\Pi}(\Pol^{\op}_{\fin} ,\Sets),
\end{equation}
where $\Funct_{\Pi}(\Pol^{\op}_{\fin} ,\Sets)$ is the category of those functors $\Pol^{\op}_{\fin}\to\Sets$ that commute with \emph{finite} products.

In \cite{La1, La2} Lawvere observed that the functor \eqref{e:second functor} and the composite functor \eqref{e:first functor} are equivalences\footnote{Informally, Lawvere's idea was to consider a ring $R$ as a set with \emph{infinitely many} operations: any $f\in\BZ [X_1,\ldots ,X_n]$ defines an operation $(x_1,\ldots ,x_n)\mapsto f(x_1,\ldots ,x_n)$, $x_i\in R$.}. He also observed that the inverse functor $\Funct_{\Pi}(\Pol^{\op} ,\Sets)\to\Rings$ takes a functor $F\in\Funct_{\Pi}(\Pol^{\op} ,\Sets)$ to the following ring $R_F\,$: as a set, $R_F=F(\BZ[x])$, and the addition (resp.~multiplication) map $R_F\times R_F\to R_F$ comes from the homomorphism $\BZ [x]\to\BZ [x]\otimes\BZ [x]$ that takes $x$ to $x\otimes 1+1\otimes x$ (resp.~to $x\otimes x$). Moreover, Lawvere observed that the word ``ring'' can be replaced by \emph{any} type of algebraic structure.

\subsection{Definition of the $(2,1)$-category of ring groupoids}  \label{ss:RGrpds definition}
\subsubsection{Notation}  \label{sss:2notation}
We keep the notation of \S\ref{sss:notation}. Let $\Grpds$ be the $(2,1)$-category of essentially small\footnote{A category si said to be essentially small if it is equivalent to a small one.} groupoids. It contains $\Sets$ as a full subcategory.

\subsubsection{Definition} \label{sss:RGrpds definition}
Let $\Funct_{\Pi}(\Pol^{\op} ,\Grpds)$ be the $(2,1)$-category of those functors $$\Pol^{\op}\to\Grpds$$ that commute with products. This $(2,1)$-category is called
the \emph{$(2,1)$-category of ring groupoids} and denoted by $\RGrpds$.

$\RGrpds$ identifies with $\Funct_{\Pi}(\Pol^{\op}_{\fin},\Grpds)$, where $\Funct_{\Pi}(\Pol^{\op}_{\fin}, \Grpds)$ is the $(2,1)$-category of those functors 
$\Pol^{\op}_{\fin}\to\Grpds$ that commute with finite products.

\subsubsection{The fully faithful functor $\Rings\mono\RGrpds$} 
By \S\ref{sss:Lawvere}, the fully faithful embedding $\Sets\mono\Grpds$ induces a fully faithful embedding
\begin{equation}   \label{e:rings are ring groupoids}
\Rings=\Funct_{\Pi}(\Pol^{\op},\Sets)\mono\Funct_{\Pi}(\Pol^{\op},\Grpds)=:\RGrpds.
\end{equation}

\subsubsection{The functor $\pi_0:\RGrpds\to\Rings$} \label{sss:pi_0 of a ring groupoid}
The set of isomorphism classes of objects of a groupoid $\Gamma$ is denoted by $\pi_0(\Gamma )$. The functor $\pi_0:\Grpds\to\Sets$ induces a functor
\[
\RGrpds:=\Funct_{\Pi}(\Pol^{\op},\Grpds)\to\Funct_{\Pi}(\Pol^{\op},\Sets)=\Rings,
\]
which will be denoted by $\pi_0:\RGrpds\to\Rings$. This functor is left adjoint to \eqref{e:rings are ring groupoids}. The unit of the adjunction provides a canonical 1-morphism $\sR\to \pi_0 (\sR )$ for any  $\sR\in\RGrpds$.

\begin{lem}    \label{l:mapping free rings to ring groupoids}
Let $\sR\in\RGrpds$ and $R\in\Pol$. Then any homomorphism $R\to\pi_0(\sR )$ lifts to a 1-morphism $R\to\sR$; moreover, this 1-morphism is unique up to 2-isomorphism.
\end{lem}

\begin{proof}
We have to show that the natural map $\pi_0 (\MMor (R,\sR))\to\Mor (R,\pi_0 (\sR))$ is bijective, where $\MMor (R,\sR)$ is the groupoid of 1-morphisms. 
By definition, $\sR$ is a product-preserving functor $F:\Pol^{\op}\to\Grpds$. By Yoneda's lemma, $\MMor (R,\sR)=F(R)$ and $\Mor (R,\pi_0 (\sR))=\pi_0(F(R))$.
\end{proof}

\subsubsection{The functor $\RGrpds\to\Grpds$}     \label{sss:RGrpds to Grpds}
The ``forgetful'' functor $\RGrpds\to\Grpds$ is defined as follows:
$F\in\RGrpds:=\Funct_{\Pi}(\Pol^{\op} ,\Grpds)$ goes to $F(\BZ [x])\in\Grpds$. One has commutative diagrams
\[
\xymatrix{
\Rings\ar[r] \ar@{^{(}->}[d] &\Sets \ar@{^{(}->}[d]  &\Rings\ar[r]  &\Sets  \\ 
\RGrpds\ar[r]&\Grpds  & \RGrpds\ar[r]\ar[u]^{\pi_0}&\Grpds\ar[u]_{\pi_0}
}
\]
whose horizontal arrows are the forgetful functors.

\subsubsection{Fiber products in $\RGrpds$}     \label{sss:fiber products}
In the $(2,1)$-category $\Grpds$ fiber products always exist. The same is true for $\RGrpds$; moreover, the functor $\RGrpds\to\Grpds$ commutes with fiber products.
Let us note that if $\sR\in\RGrpds$ and $R_1,R_2$ are usual rings equipped with morphisms to $\sR$ then $R_1\times_{\sR}R_2$ is a usual ring.

\subsection{Variants of the definition of the $(2,1)$-category of ring groupoids}
\subsubsection{Reformulation in terms of fibered categories}
As suggested to me by J.~Lurie, one could equivalently define a ring groupoid as a pair $(\cC,U)$, where
$\cC$ is a 1-category with finite coproducts and  $U: \cC\to\Pol_{\fin}$ is a functor which preserves finite coproducts and is a fibration in groupoids. To a ring groupoid $\sR$ in the sense of \S\ref{ss:RGrpds definition} one associates the pair $(\cC,U)$ defined as follows: $\cC$ is the category of 
pairs $(P,f)$, where $P\in\Pol_{\fin}$ and $f$ is a $1$-morphism $P\to\sR$; the functor $U: \cC\to\Pol_{\fin}$ forgets $f$. (This procedure is called \emph{Grothendieck construction}.)

Lurie's definition is ``elementary": it does not involve the notion of functor from an ordinary category to a 2-category. 

Recall that $\Pol_{\fin}$ was defined to be the category of those rings that are \emph{isomorphic} to $\BZ [x_1,\ldots ,x_n]$ for some $n\ge 0$.
One could replace $\Pol_{\fin}$ by its full subcategory formed by the rings $\BZ [x_1,\ldots ,x_n]$ themselves. Then Lurie's definition become even more elementary.

\subsubsection{Ring groupoids as Picard groupoids with additional structure}
Let $\mathcal Pic$ denote the symmetric monoidal $(2,1)$-category of strictly commutative Picard groupoids\footnote{The  $(2,1)$-category of strictly commutative Picard groupoids is defined in \S 1.4.1-1.4.6 of Expos\'e~XVIII of \cite{SGA4} (strictness means that for every object $X$, the commutativity isomorphism $X+X\iso X+X$ equals the identity). The symmetric monoidal structure on this $(2,1)$-category is defined in \S 1.4.8 of the same Expos\'e~XVIII.}. It is known that a ring groupoid in the sense of \S\ref{ss:RGrpds definition} is the same as a ``strictly'' commutative\footnote{By ``strict'' commutativity of the monoid $\sP$ we mean the following. First, it is commutative, so for any objects $X,Y$ of the Picard groupoid $\sP$ we have the commutativity isomorphism $X\cdot Y\iso Y\cdot X$. Second, if $X=Y$ this isomorphism is required to be the identity.} monoid in the symmetric monoidal $(2,1)$-category $\sP$. We will not use this fact.

Let us note that \S 2 of \cite{JP} contains a reformulation in ``concrete'' terms of the notion of (noncommutative) monoid in the symmetric monoidal $(2,1)$-category of (nonstrictly) commutative Picard groupoids.

\subsection{Some generalizations}
\subsubsection{Ring objects of an $n$-category}
If $\cC$ is an $n$-category with products we define a ring object in $\cC$ to be a product-preserving functor $\Pol^{\op}\to\cC$. We will 
use this definition only for $n\in\{1,2\}$, except a brief digression in \S\ref{sss:animated rings}.

Without assuming the existence of products in $\cC$, one can define a ring object in $\cC$ to be a product-preserving functor $F:\Pol^{\op}\to\Funct (\cC^{\op},\Sets)$ such that $F(\BZ [x])$ is a representable functor $\cC^{\op}\to\Sets$.

\subsubsection{Animated rings}  \label{sss:animated rings}
Let $\cC$ be the $\infty$-category of $\infty$-groupoids, which are also known as  \emph{animated sets}, see  \cite{CS}. One can also describe $\cC$ as the $\infty$-category of \emph{spaces} or \emph{simplicial sets}.

Ring objects in $\cC$ are called \emph{animated rings}, see  \cite{CS}. They form an $\infty$-category, which can also be described as the $\infty$-category of \emph{simplicial rings}.

The 2-category of groupoids is a full subcategory of the $\infty$-category $\cC$. So the 2-category of ring groupoids is a full subcategory of the 
$\infty$-category of animated rings (namely, the full subcategory of 1-truncated animated rings).

Let us note that the exposition of  ``animation'' in \cite{CS} relies on \cite[\S 5.5.8]{Lu1}.

\subsubsection{Replacing rings by other types of algebraic structure}   \label{sss:Ab instead of Rings}
One can replace rings by any type of algebraic structure (groups, Lie algebras, etc.). 
Let $\AbGrpds$ denote the analog of $\RGrpds$ obtained by replacing rings with abelian groups and replacing $\Pol$ with the category of free abelian groups.

\subsubsection{Example: Picard groupoids}   \label{sss:Picard groupoids}
It is known that the above $(2,1)$-category $\AbGrpds$ is canonically equivalent to the $(2,1)$-category of \emph{strictly} commutative Picard groupoids in the sense of \cite[Expos\'e~XVIII, \S 1.4]{SGA4}. So the
$(2,1)$-category $\AbGrpds$ is ``reasonable''. On the other hand, the bigger $(2,1)$-category of \emph{all} commutative Picard groupoids in the sense of \cite{SGA4} is no less reasonable.

\section{The naive 1-category of ring groupoids}   \label{s:the 1-category}
\subsubsection{Notation}
Let $\SSets$ (resp.~$\SRings$) be the 1-category of simplicial sets (resp.~simplicial rings). 

\subsection{Three incarnations of the 1-category of groupoids}  \label{ss:Grpds'}
\subsubsection{} \label{sss:Grpds'_1}
Let $\Grpds'_1$ be the most naive 1-category of small groupoids (its morphisms are functors on the nose). 
It contains $\Sets$ as a full subcategory. In \S\ref{sss:Grpds'_1}-\ref{sss:Grpds'_2} below we define categories $\Grpds'_2$ and $\Grpds'_3$ canonically equivalent to $\Grpds'_1$.

\subsubsection{} \label{sss:Grpds'_2}
Associating to a groupoid $\Gamma$ its nerve $N\Gamma$, one gets a fully faithful embedding $$\Grpds'_1\mono\SSets.$$
Let $\Grpds'_2\subset\SSets$ be its essential image.
A simplicial set $X$ belongs to 
$\Grpds'_2$ if and only if it has the following property: for any $n\ge 2$ and any horn $\Lambda$ in the simplex $\Delta^n$, every map $\Lambda\to X$ has one and only one extension to a map $\Delta^n\to X$. 
Passing from $\Grpds'_1$ to $\Grpds'_2$ 
is a convenient ``book-keeping device''.

\subsubsection{} \label{sss:Grpds'_3}
Here is a way to relate the 1-category $\Grpds'_1$ to the $(2,1)$-category $\Grpds$. Let $[1]$ denote the ordered set $\{ 0,1\}$ viewed as a category. Let $\Funct ([1], \Grpds)$ be the
$(2,1)$-category of functors $[1]\to\Grpds$.
Now define $\Grpds'_3\subset\Funct ([1], \Grpds)$ to be the full subcategory of functors
$\Phi :[1]\to \Grpds$ such that $\Phi (0)\in\Sets$ and the functor $\Phi (0)\to\Phi (1)$ is essentially surjective. Given $\Gamma\in\Grpds'_1$, define 
$\Phi_\Gamma\in\Grpds'_3$ as follows: $\Phi_\Gamma (1)=\Gamma$, $\Phi_\Gamma (0)$ is the set $\Ob\Gamma$ (viewed as a discrete groupoid), and the map $\Phi_\Gamma (0)\to\Phi_\Gamma (1)$ is the obvious one. Thus we get a functor
\[
\Grpds'_1\to\Grpds'_3 , \quad \Gamma\mapsto\Phi_\Gamma .
\]
This functor is an equivalence (so $\Grpds'_3$ is a 1-category rather than merely a $(2,1)$-category). The inverse functor takes $\Phi\in\Grpds'_3$ to the following groupoid $\Gamma$: the set of objects of $\Gamma$ is $\Phi (0)$, and for every $x,y\in \Phi (0)$ one has $\Mor_\Gamma (x,y):=\Mor_{\Phi (1)} (x_1,y_1)$, where $x_1,y_1\in\Phi (1)$ are the images of $x$ and $y$ (composition of $\Gamma$-morphisms comes from composition of $\Phi (1)$-morphsims).

\subsection{Three incarnations of the 1-category of ring groupoids}  \label{ss:RGrpds'}
For $n=1,2,3$ define $\RGrpds'_n$ to be the category of ring objects in the category $\Grpds'_n\,$.
Similarly to \S\ref{sss:RGrpds to Grpds}, the ``forgetful'' functor $\RGrpds'_n\to\Grpds'_n$ is defined as follows:
$F\in\RGrpds'_n:=\Funct_{\Pi}(\Pol^{\op} ,\Grpds'_n)$ goes to $F(\BZ [x])\in\Grpds'_n$.

\subsubsection{On $\RGrpds'_2$}
By \S\ref{sss:Lawvere}, $\RGrpds'_2$ identifies with the category of simplicial rings such that the underlying simplicial set is a nerve of a groupoid.

\subsubsection{A convenient way to think of $\RGrpds'_2$}
For any category $\cC$, there is a notion of groupoid internal to $\cC$. 
The category $\RGrpds'_2$ identifies with the category of groupoids internal to $\Rings$. We will mostly think of $\RGrpds'_2$ in this way.

The forgetful functor $\Rings\to\Sets$ commutes with projective limits, so a groupoid internal to $\Rings$ is just a usual groupoid $\Gamma$ plus a ring structure on $\Ob\Gamma$ and on $\Mor\Gamma$ (where $\Mor\Gamma$ is  the set of \emph{all\,} morphisms in $\Gamma$) such that the following maps are ring homomorphisms:

(i) the map $\Mor\Gamma\to\Ob\Gamma$ that takes a morphism to its source (resp.~target);

(ii) the map $\Ob\Gamma\to\Mor\Gamma$ that a takes $a\in\Ob\Gamma$ to $\id_a\,$;

(iii) the map $(f,g)\mapsto g\circ f$, which is defined on the ring of all composable pairs of $\Gamma$-morphisms.

\subsubsection{On $\RGrpds'_3$}  \label{sss:RGrpds'_3}
The category $\RGrpds'_3$ identifies with the category formed by functors $\Phi :[1]\to \RGrpds$ such that $\Phi (0)\in\Rings$ and the functor $\Phi (0)\to\Phi (1)$ is essentially surjective
(by this we mean that it is essentially surjective as a functor between ``plain'' groupoids).

\begin{prop}   \label{p:essential surjectivity}
The functor $\RGrpds'_1\to\RGrpds$ is essentially surjective.
\end{prop}

\begin{proof}
It suffices to show that the functor $\RGrpds'_3\to\RGrpds$ is essentially surjective. By \S\ref{sss:RGrpds'_3}, the problem is to show that for any ring groupoid $\sR$ there exists an essentially surjective 1-morphism $R\to\sR$, where $R$ is a ring.

Choose a polynomial ring $R$ and an epimorphism $R\epi\pi_0(\sR )$. By Lemma~\ref{l:mapping free rings to ring groupoids}, it lifts to a 1-morphism $R\to\sR$. The latter is essentially surjective. 
\end{proof}

\subsection{Quasi-ideals, DG rings, and 1-truncated simplicial rings}  \label{ss:quasi-ideals}
In this subsection we define a category $\cQ$ and some categories which are obviously equivalent to it.
In \S\ref{ss:Q=RGrpds'} we will show that these categories are equivalent to the categories
$\RGrpds'_1$, $\RGrpds'_2$, $\RGrpds'_3$ from \S\ref{ss:RGrpds'}; in other words, they can be considered as incarnations of the 1-category of ring groupoids.
These incarnations are more manageable than those from  \S\ref{ss:RGrpds'}.

\subsubsection{Quasi-ideals and the category $\cQ$} \label{sss:quasi-ideals} 
By a \emph{quasi-ideal} in a ring $C$ we mean a pair $(I,d)$, where $I$ is a $C$-module and $d:I\to C$ is a $C$-linear map such that 
\begin{equation}  \label{e:quasi-ideal}
d(x)\cdot y=d(y)\cdot x
\end{equation}
for all $x,y\in I$. 

Let $\cQ$ be the category of all triples $(C,I,d)$, where $C$ is a ring and $(I,d)$ is a quasi-ideal in~$C$.

\subsubsection{Remarks}  \label{sss:quasi-ideal rems}
(i) A quasi-ideal $(I,d:I\to C)$ with $\Ker d=0$ is essentially the same as an ideal in $C$.

(ii) If  $(I,d)$ is a quasi-ideal in $C$ then $I$ is a (non-unital) ring with respect to the multiplication operation 
$(x,y)\mapsto d(x)\cdot y$.

\subsubsection{Quasi-ideals and DG rings}  \label{sss:Q & DGRings}
If $(I,d)$ is a quasi-ideal in a ring $C$ then one can define a DG ring $R$ as follows: $R^0=C$, $R^{-1}=I$, $R^i=0$ for $i\ne 0,1$, the differential in $R$ is given by $d:I\to C$, and the multiplication maps 
$$R^0\times R^0\to R^0, \quad R^0\times R^{-1}\to R^{-1}$$ 
come from the ring structure on $C$ and the $C$-module structure on $I$; note that the Leibnitz rule in $R$ is equivalent to \eqref{e:quasi-ideal}. Thus one gets an equivalence between the category $\cQ$ from \S\ref{sss:quasi-ideals} and the category of DG rings $R$ such that $R^i=0$ for $i\ne 0,-1$.

\subsubsection{1-truncated simplicial sets and rings}
Recall that a simplicial set is a functor $$\Delta^{\op}\to\Sets,$$ where $\Delta$ is the category of finite linearly ordered sets.
By a \emph{1-truncated simplicial set} we mean a functor $\Delta_{\le 1}^{\op}\to\Sets$, where $\Delta_{\le 1}\subset\Delta$ is the full subcategory formed by linearly ordered sets of order $\le 2$.  The category of 1-truncated simplicial sets will be denoted by $\SSets_{\le 1}$. We have the restriction (a.k.a. truncation) functor 
$\SSets\to\SSets_{\le 1}$.

Similarly, we have the category of \emph{1-truncated simplicial rings,} denoted by $\SRings_{\le 1}$, and the truncation functor $\SRings\to\SRings_{\le 1}$.
Explicitly, an object of $\SRings_{\le 1}$ is a collection
\[
(A_0,\, A_1,\,  \partial_0:A_1\to A_0,\, \partial_1:A_1\to A_0,\,  s:A_0\to A_1),
\]
where $A_0,A_1$ are rings and $\partial_0,\partial_1,s$ are ring homomorphisms such that 
\[
\partial_0\circ s=\partial_1\circ s=\id_{A_0}\, .
\]

\subsubsection{The subcategory $\SRings_{\le 1}^{\good}$}  \label{sss:goodness}
Let $\SRings_{\le 1}^{\good}\subset\SRings_{\le 1}$ be the full subcategory of collections 
$(A_0,\, A_1,\partial_0,\partial_1, s)\in\SRings_{\le 1}$ such that
\begin{equation}  \label{e:KerKer=0}
(\Ker\partial_0 ) \cdot  (\Ker\partial_1 )=0.
\end{equation}

\begin{lem}   \label{l:Q and SRings_le 1}
The categories $\cQ$ and $\SRings_{\le 1}^{\good}$ are equivalent.
\end{lem}

\begin{proof}
Given $(A_0,\, A_1,\partial_0,\partial_1, s)\in\SRings_{\le 1}^{\good}$, set $C:=A_0$, $I:=\Ker\partial_0$, define $d:I\to C$ by $d:=\partial_1|_I$, and define the $C$-module structure on $I$ using the homomorphism $C=A_0\overset{s}\longrightarrow A_1$. Let us prove \eqref{e:quasi-ideal}. If $x,y\in I$ then $x(y-s(\partial_1(y))=0$ by \eqref{e:KerKer=0}, which means that the r.h.s. of \eqref{e:quasi-ideal} equals $xy$. By symmetry, this is also true for the l.h.s. of \eqref{e:quasi-ideal}.

We have constructed a functor $\SRings_{\le 1}^{\good}\to\cQ$. It is easy to see that it is an equivalence and the inverse functor $\cQ\to\SRings_{\le 1}^{\good}$ is as follows: $A_0:=C$, $A_1:=C\oplus I$, $s:A_0\to A_1$ is the inclusion $C\mono C\oplus I$, the maps $\partial_0,\partial_1:C\oplus I\to C$ are given by
\[
\partial_0 (c,x):=c, \;\partial_1 (c,x):=c+dx, \quad c\in C, x\in I,
\]
and the ring structure on $A_1$ comes from the $C$-module structure on $I$ and the operation on $I$ defined in \S\ref{sss:quasi-ideal rems}(ii); one checks that
$\partial_1$ is a ring homomorphism and \eqref{e:KerKer=0} holds.
\end{proof}

\subsubsection{Remark}
We will always use the equivalence $\cQ\iso\SRings_{\le 1}^{\good}$ constructed in the above proof. One can check that it is isomorphic to the equivalence 
$\cQ\iso\SRings_{\le 1}^{\good}$ that one gets by reversing the roles of $\partial_0$ and $\partial_1$.

\subsection{The equivalence $\cQ\iso\RGrpds'_2\,$}  \label{ss:Q=RGrpds'}
We will first construct a canonical equivalence $\cQ\iso\RGrpds'_2$ using the category $\SRings_{\le 1}^{\good}$ as an intermediate step. 
Then we describe this equivalence directly, see \S\ref{sss:description of Cone (d)}. The reader may prefer to look at \S\ref{sss:description of Cone (d)} before reading 
\S\ref{sss:1-truncated nerve}-\ref{sss:Q=RGrpds'_2}.

\subsubsection{The 1-truncated nerve of a category}   \label{sss:1-truncated nerve}
The nerve of a category $\cC$ is denoted by $N\cC$. The image of $N\cC$ under the functor $\SSets\to\SSets_{\le1}$ will be denoted by $N_{\le 1}\cC$ and called
the \emph{1-truncated nerve} of $\cC$.

In other words, $N_{\le 1}\cC$ remembers the set of objects of $\cC$, the set of morphisms in $\cC$, the source and target of each morphism, and the morphisms $\id_c$ for all $c\in\Ob\cC$; however, it forgets the composition of morphisms. Thus $N_{\le 1}\cC$ is not really interesting.

Similarly, if $\cC$ is a category internal to $\Rings$ one has the nerve $N\cC\in\SRings$ and the 1-truncated nerve $N_{\le 1}\cC\in\SRings_{\le1}$. The next proposition shows that in this setting $N_{\le 1}\cC$ is quite interesting. 

\begin{prop}   \label{p:categories internal to Rings}
(i) The above functor
\[
N_{\le 1}:\{\mbox{Categories internal to $\Rings$\}}\to\SRings_{\le1}
\]
is fully faithful.

(ii) Its essential image equals $\SRings_{\le 1}^{\good}$.

(iii) Every category internal to $\Rings$ is a groupoid.
\end{prop}

The proposition will be deduced from Lemmas~\ref{l:category gives 1-truncated Sset}-\ref{l:goodness appears}. The first two of them 
describe  categories internal to $\Ab$, where $\Ab$ is the category of abelian groups.

\begin{lem}   \label{l:category gives 1-truncated Sset}
Let $\cC$ be a category internal to $\Ab$ and let $(A_0,\, A_1,\partial_0,\partial_1, s)$ be its 1-truncated nerve.

(i) Suppose that $f,g\in A_1$ and $\partial_1 (f)=\partial_0 (g)$ (in other words, $f$ and $g$ form a composable pair of morphisms). Then
\begin{equation}  \label{e:circ in terms of +}
g\circ f=f+g-s(a), \quad \mbox{where }a=\partial_1 (f)=\partial_0 (g).
\end{equation}

(ii) $\cC$ is a groupoid. The inverse of $f\in A_1$ equals $s(\partial_0(f))+s(\partial_1 (f))-f$.
\end{lem}

\begin{proof}
Let $B$ be the group of all pairs $(f,g)$ as in (i). The map 
\[
B\to A_1, \quad (f,g)\mapsto g\circ f
\]
is a group homomorphism by the definition of ``category internal to $\Ab$''. The map $B\to A_1$ given by \eqref{e:circ in terms of +} is also a group homomorphism.
The group $B$ is generated by $B_1$ and $B_2$, where $B_1$ (resp. $B_2$) is the group of all $(f,g)\in B$ such that $f$ (resp.~$g$) is an identity morphism. So it suffices to check \eqref{e:circ in terms of +} if either $f=\id_a$ or $g=\id_a$. This is clear because $s(a)$ is just another name for $\id_a$.

We have proved (i). Statement (ii) follows.
\end{proof}

Similarly to Lemma~\ref{l:category gives 1-truncated Sset}(i), one proves the following converse statement.

\begin{lem} \label{l:1-truncated Sset gives category}
Let $(A_0,\, A_1,\partial_0,\partial_1, s)$ be a 1-truncated simplicial abelian group. Then the operation \eqref{e:circ in terms of +} makes it into a category internal to $\Ab$.\qed
\end{lem} 

Proposition~\ref{p:categories internal to Rings} follows from Lemmas~\ref{l:category gives 1-truncated Sset}-\ref{l:1-truncated Sset gives category} and the following one.

\begin{lem} \label{l:goodness appears}
Let $(A_0,\, A_1,\partial_0,\partial_1, s)$ be a 1-truncated simplicial ring. Let
\[
B:=\{(f,g)\,|\, f,g\in A_1,\,\partial_1 (f)=\partial_0 (g) \}.
\]
In this situation, the map $B\to A_1$ defined by \eqref{e:circ in terms of +} is a ring homomorphism if and only if $(\Ker\partial_0 ) \cdot  (\Ker\partial_1 )=0$.
\end{lem} 

\begin{proof}
Let $\varphi :B\to A_1$ be the map \eqref{e:circ in terms of +}. Let $B_1,B_2$ be as in the proof of Lemma~\ref{e:circ in terms of +}. Then $\varphi |_{B_1}$ and $\varphi |_{B_2}$ are automatically ring homomorphisms. One has
\[
B_1=(B_1\cap B_2)+J_1, \quad B_2=(B_1\cap B_2)+J_2,
\]
where $J_1:=\{(0,g)\,|\, g\in \Ker\partial_0 \}$,\, $J_2:=\{(f,0)\,|\, f\in\Ker\partial_1 ) \}$. Moreover, $J_1\cdot J_2=0$. 
So $\varphi :B\to A_1$  is a ring homomorphism if and only if $\varphi(J_1)\cdot \varphi(J_2)=0$, which means that
$(\Ker\partial_0 ) \cdot  (\Ker\partial_1 )=0$.
\end{proof}

\subsubsection{The equivalence $\cQ\iso\RGrpds'_2\,$}    \label{sss:Q=RGrpds'_2}
Recall that $\RGrpds'_2$ is the category of groupoids internal to $\Rings$, and $\cQ$ is the category of all triples $(C,I,d)$, where $C$ is a ring and $(I,d:I\to C)$ is a quasi-ideal in~$C$. In the proof of  Lemma~\ref{l:Q and SRings_le 1} we constructed an equivalence $\cQ\iso\SRings_{\le 1}^{\good}$. Composing it with the equivalence 
$\SRings_{\le 1}^{\good}\iso\RGrpds'_2$ from Proposition~\ref{p:categories internal to Rings}, we get an equivalence
$$\cQ\iso\RGrpds'_2\,.$$
We denote it as follows:
\[
(C,I,d)\mapsto\Cone (d)=\Cone(I\overset{d}\longrightarrow C).
\]
Let us now give a description of $\Cone (d)$  (the reader may prefer to use it as a definition). 

\subsubsection{Explicit description of $\Cone (d)$}   \label{sss:description of Cone (d)}
Let $C$ be a ring and $(I,d:I\to C)$ a quasi-ideal in~$C$. Then $\Cone(I\overset{d}\longrightarrow C)$ is a groupoid internal to $\Rings$, whose set of objects is the ring $C$ and whose morphisms are labeled by $C\times I$. The morphism corresponding to $c\in C$ and $x\in I$ is a morphism $c\to c+dx$, denoted by $f_{c,x}\, $. Morphisms are composed as follows:
\[
f_{c+dx,y}\circ f_{c,x}=f_{c,x+y}\,.
\]
Finally, the ring structure on the set of morphisms is given by
\[
f_{c,x}+f_{c',x'}=f_{c+c',x+x'}\,,
\]
\begin{equation}  \label{e:product of morphisms}
f_{c,x}\cdot f_{c',x'}=f_{cc',y}\, , \mbox{ where } y=cx'+c'x+x\cdot dx'=cx'+c'x+x'\cdot dx.
\end{equation}

Note that as a usual groupoid (rather than a groupoid internal to $\Rings$), $\Cone (d)$ is just the \emph{quotient groupoid\,} of $C$ by the following action of $I$: an element $x\in I$ takes $c\in C$ to~$c+dx$.

\subsubsection{Quasi-isomorphisms in $\cQ$}   \label{sss:q-isomorphisms in Q}
Let $f:(C,I,d)\to (C',I',d')$ be a morphism in $\cQ$. It induces a functor $\Cone (d)\to\Cone (d')$ between the corresponding groupoids. Using \S\ref{sss:description of Cone (d)}, one checks that this functor is an equivalence if and only if $f$ is a quasi-isomorphism (which means that $f$ induces isomorphisms $\Ker d\iso\Ker d'$ and $\Coker d\iso\Coker d'$).

\subsubsection{Motivation of the $\Cone$ notation}
If $C$ is any ring then $\Cone(0\to C)$ identifies with~$C$ (viewed as discrete groupoid).
In general, the morphism $(C,I,d)\to (C/d(I),0,0)$ induces a functor
\[
\Cone(I\overset{d}\longrightarrow C)\to\Cone (0\to  C  /d(I))=C/d(I),
\]
and if $\Ker d=0$ this functor is an equivalence (but not an isomorphism, unless $I=0$).

Let us note that in the context of abelian groups (instead of rings and quasi-ideals) the groupoid $\Cone (d)$ is considered in \cite[Expos\'e XVIII, \S 1.4]{SGA4}, where it is denoted by ${\rm ch}(d)$. It is proved there that the $(2,1)$-category of strictly commutative Picard groupoids\footnote{In our language, this is the $(2,1)$-category $\AbGrpds$, see \S\ref{sss:Ab instead of Rings}-\ref{sss:Picard groupoids}.} is canonically equivalent to the full subcategory of the DG category of complexes of abelian groups formed by complexes with cohomology concentrated in degrees $-1$ and $0$; moreover, this equivalence takes ${\rm ch}(d)$ to the usual cone of $d$ (i.e., to the complex $0\to I\overset{d}\longrightarrow C\to 0$ placed in degrees $-1$ and $0$). This is our main motivation for writing $\Cone (d)$
instead of ${\rm ch}(d)$.

\subsection{The parallel story for groups}
This subsection and \S\ref{ss:Eilenberg-Zilber} can be skipped by the reader.

\subsubsection{Abelian groups}
By Lemmas~\ref{l:category gives 1-truncated Sset}-\ref{l:1-truncated Sset gives category}, the 1-category of categories internal to $\Ab$ (or equivalently, groupoids internal to $\Ab$)  identifies with the category of all 1-truncated simplicial abelian groups. Similarly to Lemma~\ref{l:Q and SRings_le 1}, the latter identifies with the category of triples $(C,I,d)$, where $C,I\in\Ab$ and $d:I\to C$ is a homomorphism (this is a ``baby version'' of the Dold-Kan equivalence).

\subsubsection{Arbitrary groups}
Let $\Groups$ denote the category of all groups. Let $\SGroups_{\le 1}^{\good}$ be the category of 1-truncated simplicial groups 
$(G_0,\, G_1,\partial_0,\partial_1, s)$ such that $\Ker\partial_0$ centralizes $\Ker\partial_1$ (this condition is somewhat similar to \eqref{e:KerKer=0}).
One can check that Proposition~\ref{p:categories internal to Rings} and Lemma~\ref{l:Q and SRings_le 1} remain valid if one replaces $\Rings,\SRings_{\le 1}^{\good}$ by $\Groups, \SGroups_{\le 1}^{\good}$ and also replaces the category $\cQ$ from \S\ref{sss:quasi-ideals} by the category of \emph{crossed modules.}

The notion of crossed module is due to J.~H.~C.~Whitehead \cite{Wh}. For an overview of it, see \cite[\S 6.6.12]{Wei}, \cite{N1}, and references therein.

\subsection{Remarks on DG rings}   \label{ss:Eilenberg-Zilber}
This subsection can be skipped by the reader.

\subsubsection{Notation}
Let $\DGRings^{\le 0}$ (resp.~$\DGRings^{0,-1}$) be the category of DG rings $R$ such that $R^i=0$ for $i>0$ (resp. for $i\ne 0,-1$).

\subsubsection{DG rings via Eilenberg-Zilber} 
The normalized chain complex of a simplicial ring has a DG ring structure, which is defined via the Eilenberg-Zilber map. 
One can check that the functor $\cQ\to\DGRings^{\le0}$ from \S\ref{sss:Q & DGRings} is isomorphic to the composite functor
\[
\cQ\to\RGrpds'_2\mono\SRings\overset{\cN}\longrightarrow\DGRings^{\le0} ,
\]
where the first arrow is as in \S\ref{sss:Q=RGrpds'_2}, the second one takes a groupoid to its nerve, and $\cN$ is the functor of normalized chains. We will not use this fact.

\subsubsection{Simplicial rings via Alexander-Whitney} 
Recall that the functor $\cQ\to\DGRings^{\le0}$ induces an equivalence $\cQ\iso\DGRings^{0,-1}\subset\DGRings^{\le0}$. Let us discuss the composite functor
\begin{equation}   \label{e:DGRings to SRings}
\DGRings^{0,-1}\iso\cQ\to\RGrpds'_2\mono\SRings.
\end{equation}

Let $\bf Rings$  be the category of unital associative but not necessarily commutative rings. The functor $\cN :\SRings\to\DGRings^{\le0}$ extends to a functor
$$\cN :\rm S\bf Rings\to\rm DG\bf Rings^{\le 0}.$$ 
The latter has a canonical right inverse $\Gamma :\rm DG\bf Rings^{\le 0}\to S\bf Rings$, which is defined using the cup product on the cochain complexes of certain simplicial sets (i.e., using the Alexander-Whitney map). 
The functor $\Gamma$ does not preserve commutativity (because the $\cup$-product is not commutative at the level of cochains); in other words,
\begin{equation}  \label{e: no inclusion}
\Gamma (\DGRings^{\le0} )\not\subset\SRings.
\end{equation}
However, one can check that 
\begin{equation}  \label{e: inclusion holds}
\Gamma (\DGRings^{0,-1} )\subset\SRings, 
\end{equation}
and the functor \eqref{e:DGRings to SRings} is isomorphic to
$\Gamma :\DGRings^{0,-1} \to\SRings$. We will not use this fact.

Here is a way to believe in formulas \eqref{e: no inclusion}-\eqref{e: inclusion holds} (or even to prove them): in formula \eqref{e:product of morphisms} we have
$x\cdot dx'=x'\cdot dx$ because $x'\cdot dx-x\cdot dx'=d(xx')$ and $xx'=0$ (thus the condition $xx'=0$ is very essential).

\section{The Aldrovandi-Noohi model of the $(2,1)$-category $\RGrpds$} \label{s:RGrpds_AN}
In this section we define a $(2,1)$-category $\RGrpds_\AN$. and prove Theorem~\ref{t:main}, which says that $\RGrpds_\AN$ is canonically equivalent to $\RGrpds$, i.e., to the $(2,1)$-category of ring groupoids defined in \S\ref{s:the 2-category}. This theorem is a variant of the main result of \cite{Al2} (the main difference is that in \cite{Al2} rings are not assumed to be commutative).

The definition of $\RGrpds_\AN$ follows the ideas\footnote{The only difference is that Noohi  \cite{N0} considers (noncommutative) groups rather than rings and crossed modules rather than objects of $\DGRings^{0,-1}$.} of B~Noohi \cite{N0}, which were further developed in the works by E.~Aldrovandi and B.~Noohi 
\cite{AN1,AN2, Al1,Al2, N0,N2}. The symbol $\AN$ stands for Aldrovandi-Noohi and also for \emph{``anamorphism''}(see \S\ref{sss:Three classes}) and \emph{``anafunctor''} (the latter notion will be recalled in~\S\ref{s:anafunctors}).

\subsection{The 2-category of correspondences}  \label{ss:correspondences}
Let $\cC$ be a category in which 
finite fiber products always exist. Then one defines the 2-category of correspondences $\Corr (\cC )$ as follows.

(i) The objects of $\Corr (\cC )$ are those of $\cC$.

(ii) For $c_1,c_2\in\cC$, the category of $\Corr (\cC )$-morphisms is defined to be the category of 
diagrams\footnote{A morphism from a diagram $c_1\overset{f}\longleftarrow c_{12}\overset{g}\longleftarrow c_2$ to a diagram $c_1\overset{f'}\longleftarrow c'_{12}\overset{g'} \longleftarrow c_2$ is a morphism  $h:c_{12}\to c'_{12}$ such that $f'h=f$ and $g'h=g$.} $c_1\leftarrow c_{12}\to c_2$ in $\cC$.
This category is denoted by
$\Corr (c_1,c_2)$, and its objects are called \emph{correspondences from $c_1$ to $c_2\,$}. 

(iii) The composition of correspondences $c_1\leftarrow c_{12}\to c_2$ and $c_2\leftarrow c_{23}\to c_3$ is defined to be the correspondence
$c_1\leftarrow c_{12}\times_{c_2} c_{23}\to c_3\,$.

Let us note that correspondences are also called \emph{spans} (e.g., in \cite[\S 9.1]{N0}).

\subsection{Correspondences in $\DGRings^{0,-1}$}   \label{ss: Corr in DGRings}
\subsubsection{Recollections on $\DGRings^{0,-1}$}
Recall that $\DGRings^{0,-1}$ stands for the category of DG~rings $R$ such that $R^i=0$ for $i\ne 0,-1$. This category is one of the incarnations of the 1-category of ring groupoids (see \S\ref{sss:Q & DGRings} and \S\ref{ss:Q=RGrpds'}).

\subsubsection{Three classes of correspondences}  \label{sss:Three classes}
Let $R_1,R_2\in\DGRings^{0,-1}$. According to \S\ref{ss:correspondences}, a correspondence from $R_1\to R_2$ is just a diagram
\begin{equation}   \label{e:correspondence in DGRings}
R_1\overset{f}\longleftarrow R_{12}\overset{g}\longrightarrow R_2
\end{equation}
in~$\DGRings^{0,-1}$.

We say that a correspondence \eqref{e:correspondence in DGRings} is \emph{admissible} (resp.~\emph{weakly admissible}) if $f$ is a quasi-isomorphism and the map
$R_{12}^{-1}\to R_{1}^{-1}\times R_{2}^{-1}$ is an isomorphism (resp.~epimorphism). We say that a \eqref{e:correspondence in DGRings} is an \emph{anamorphism}\footnote{Anamorphisms are analogous to \emph{anafunctors}, see \S\ref{sss:Anafunctors} below.} from $R_1$ to $R_2$ if $f$ is a surjective\footnote{The class of \emph{surjective} quasi-isomorphisms is good for us because it is closed under pullbacks (i.e., if $f:R'\to R$ is a surjective quasi-isomorphism then so is $\tilde f:R'\times_R\tilde R\to\tilde R$). Without surjectivity this would be false. See also \S\ref{sss:composition remarks}(i) below.} quasi-isomorphism.

Let $\Corr_\adm (R_1,R_2)$ (resp.~$\Corr_\wadm (R_1,R_2)$) be the category of admissible (resp.~weakly admissible) correspondences from $R_1$ to $R_2$.
Let $\Corr_\ana (R_1,R_2)$ be the category of anamorphisms from $R_1$ to $R_2$. Then
$$\Corr_\adm (R_1,R_2)\subset\Corr_\wadm (R_1,R_2)\subset\Corr_\ana (R_1,R_2).$$

\subsubsection{Admissible correspondences via ``butterflies''}   \label{sss:butterflies}
Admissible correspondences have the following description, which I learned from the works by Aldrovandi and Noohi (e.g., see \cite[Def.~8.1]{N0} or \cite[\S 4.1.3]{AN1}). Note that
given an admissible correspondence \eqref{e:correspondence in DGRings}, one can use the isomorphism  $R_{12}^{-1}\iso R_{1}^{-1}\times R_{2}^{-1}$ to write 
$d:R_{12}^{-1}\to R_{12}^0$ as a pair of maps $h_i:R_i^{-1}\to R_{12}^0$ for~$i=1,2$. Thus we see that an admissible correspondence \eqref{e:correspondence in DGRings} is the same as a commutative diagram
\begin{equation}   \label{e:butterfly}
  \vcenter{
  \xymatrix@=1pc{
    R_1^{-1}\ar[dd]_d \ar[dr]^{h_1}  & &R_2^{-1} \ar[dl]_{h_2} \ar[dd]^d\\
    & R_{12}^0\ar[dl]_{f^0} \ar[dr]^{g^0} &  \\
    R_1^0 & & R_2^0
  }}
\end{equation}
with the following properties:

(i) $R_{12}^0$ is a ring, and $f^0,g^0$ are ring homomorphisms;

(ii) $h_1$ and $h_2$ are $R_{12}^0$-module homomorphisms assuming that the $R_{12}^0$-module structure on $R_1^{-1}$ (resp.~$R_2^{-1}$) is defined via
$f^0$ (resp.~$g^0$);

(iii) the NW-SE sequence in \eqref{e:butterfly} is a complex, and the NE-SW sequence is exact, i.e., 
\[
\Ker h_2=0 ,\quad \im h_2=\Ker f^0, \quad \Coker f^0=0.
\]
A commutative diagram \eqref{e:butterfly} with properties (i)-(iii) is called a \emph{butterfly}.

\begin{lem}   \label{l:Corr_adm is a groupoid}
The category $\Corr_\adm (R_1,R_2)$ is a groupoid.
\end{lem}

\begin{proof}
We have to show that any morphism in $\Corr_\adm (R_1,R_2)$ is an isomorphism. 
By \S\ref{sss:butterflies}, it suffices to prove a similar statement for morphisms of butterflies. This is clear because the NE-SW sequence in \eqref{e:butterfly} is exact.
\end{proof}

\begin{lem}   \label{l:composition of weakly admissible}
The composition of weakly admissible correspondences is weakly admissible. The composition of anamorphisms is an anamorphism. \qed
\end{lem}

\subsubsection{Remarks}   \label{sss:composition remarks}
(i) The class of correspondences $R_1\overset{f}\longleftarrow R_{12}\to R_2$ such that $f$ is a quasi-isomorphism is \emph{not} closed under composition of correspondences.

(ii) The composition of admissible correspondences is  \emph{not admissible,} in general. To cure this, one uses the \emph{admissibilization functor} constructed in the next subsection.

\subsection{Admissibilization}   \label{ss:Admissibilization}
\begin{prop}     \label{p:Adm}
The inclusion functor $\Corr_\adm (R_1,R_2)\mono\Corr_\ana (R_1,R_2)$ has a left adjoint functor $\Adm :\Corr_\ana (R_1,R_2)\to\Corr_\adm (R_1,R_2)$.
\end{prop}

 The functor $\Adm :\Corr_\ana (R_1,R_2)\to\Corr_\adm (R_1,R_2)$ is called \emph{admissibilization.} Let us note that the restriction of $\Adm$ to 
 $\Corr_\wadm (R_1,R_2)\subset\Corr_\ana (R_1,R_2)$ has a very simple description, see Lemma~\ref{l:weakly admissible}(iv) below.

\begin{proof}
Let $R_1\overset{f}\longleftarrow R_{12}\overset{g}\longrightarrow R_2$ be an anamorphism. Let $L:=R_1\times R_2.$. Let $\varphi:R_{12}\to L$ be given by $(f,g)$. Note that since $f$ is a quasi-isomorphism\footnote{This is the only place in the proof where we use that $f$ is an anamorphism. In particular, we do not use surjectivity of $f$.}, we have
\begin{equation}  \label{e:injectivity on H_1}
\Ker (H^{-1}(R_{12})\to H^{-1}(L))=0,
\end{equation}
(as usual, $H^i$ stands for the $i$-th cohomology).

Let $\fC$ be the category of factorizations of $\varphi$ as
\begin{equation}  \label{e:factorization}
R_{12}\overset{\psi}\longrightarrow \tilde R_{12}\overset{\chi}\longrightarrow L
\end{equation}
such that the correspondence
\begin{equation}   \label{e:correspondnece to construct}
R_1\leftarrow \tilde R_{12}\to R_2
\end{equation}
given by $\chi$ is admissible. By Lemma~\ref{l:Corr_adm is a groupoid}, $\fC$ is a groupoid. We have to show that $\fC$ is a point.

Admissibility of \eqref{e:correspondnece to construct} is equivalent to the following properties of \eqref{e:factorization}: $\psi$ is a quasi-isomorphism, and 
$\chi^{-1}:\tilde R_{12}^{-1}\to L^{-1}$ is an isomorphism. So  as a complex of abelian groups, 
$\tilde R_{12}$ has to be as follows: $\tilde R_{12}^{-1}=L^{-1}$, and the pair
$(\tilde R_{12}^0, \tilde d:R_{12}^{-1}\to R_{12}^0)$ is determined by the push-out diagram
\[
\xymatrix{
R_{12}^{-1}\ar[r]^{\varphi^{-1}} \ar[d]_d & \;\; L^{-1}\ar[d]^{\tilde d}&\!\!\!\!\!\!\!\!\!\!\!\!\!\!\!\!=\tilde R_{12}^{-1}\\
R_{12}^{0}\ar[r]^{\psi^0} & \tilde R_{12}^{0}
}
\]
The morphism $R_{12}\to\tilde R_{12}$ defined by this diagram is a quasi-isomorphism by \eqref{e:injectivity on H_1}.

Now we have a commutative diagram of complexes of abelian groups 
\begin{equation}     \label{e:diagram to enhance}
\xymatrix{
R_{12}^{-1}\ar[r]^{\varphi^{-1}} \ar[d]_d &  \tilde R_{12}^{-1} \ar[r]^{\id}_{\sim}\ar[d]^{\tilde d}& \;\; L^{-1}\ar[d]^{d'}\\
R_{12}^{0}\ar[r]^{\psi^0} & \tilde R_{12}^{0}\ar[r]^{\varphi^0} &L^0
}
\end{equation}
and the problem is to define a ring structure on $\tilde R_{12}$ which makes \eqref{e:diagram to enhance} into a diagram of DG~rings. 
This problem has at most one solution because $ \tilde R_{12}^{0}=R_{12}^{0}+ \tilde d(\tilde R_{12}^{-1} )$ and from the right square of \eqref{e:diagram to enhance} we see that for
$x,y\in \tilde R_{12}^{-1}$ we must have
\begin{equation}
(\tilde d x)\cdot y=(d' x)\cdot y, \quad (\tilde d x)\cdot (\tilde d y)=\tilde d((\tilde d x)\cdot y)=\tilde d((d' x)\cdot y).
\end{equation}
It is straightforward to check that these formulas indeed define a solution to our problem. 
\end{proof}

\subsubsection{Admissibilization as localization}   \label{sss:Admissibilization as localization}
The functor $$\Adm :\Corr_\ana (R_1,R_2)\to\Corr_\adm (R_1,R_2)$$ identifies $\Corr_\adm (R_1,R_2)$ with the groupoid obtained from $\Corr_\ana (R_1,R_2)$ by inverting all morphisms. This follows from Lemma~\ref{l:Corr_adm is a groupoid} and the adjunction from Proposition~\ref{p:Adm}.

\bigskip

The admissibilization of a \emph{weakly admissible} correspondence is described in part~(iv) of the next lemma.

\begin{lem}   \label{l:weakly admissible}
Suppose that a correspondence 
\begin{equation}   \label{e:2correspondence in DGRings}
R_1\overset{f}\longleftarrow R_{12}\overset{g}\longrightarrow R_2
\end{equation}
is weakly admissible. Then

(i) $f$ is surjective;

(ii) $\Ker f$ is acyclic;

(iii) the DG ideal $I\subset R_{12}$ generated by $\Ker (R_{12}^{-1}\to R_{1}^{-1}\times R_{2}^{-1})$ is acyclic;

(iv) the correspondence
\[                                     
R_1\leftarrow R_{12}/I\to R_2
\]
is the admissibilization of \eqref{e:2correspondence in DGRings} \qed
\end{lem}

\subsubsection{Admissibilization of the composition}   \label{sss:Adm of composition} 
Let $\alpha\in\Corr_\ana (R_1,R_2)$, $\beta\in\Corr_\ana (R_2,R_3)$, where $R_1,R_2,R_3\in \DGRings^{0,-1}$.
The canonical morphisms $\alpha\to\Adm (\alpha )$ and $\beta\to\Adm (\alpha )$ induce a morphism $\beta\circ\alpha\to\Adm (\beta )\circ\Adm (\alpha )$. The corresponding morphism
\[
\Adm (\beta\circ\alpha)\to\Adm (\Adm (\beta)\circ\Adm (\alpha ))
\]
is an isomorphism by Lemma~\ref{l:Corr_adm is a groupoid}. Thus $\Adm (\beta\circ\alpha)=\Adm (\Adm (\beta)\circ\Adm (\alpha ))$.

\subsubsection{The map $\adm:\Hom (R_1,R_2)\to\Corr_\adm (R_1,R_2)$}    \label{sss:Spl}
Define a map 
\begin{equation}   \label{e:Spl}
\adm:\Hom (R_1,R_2)\to\Corr_\adm (R_1,R_2)
\end{equation}
to be the composition
$\Hom (R_1,R_2)\to\Corr_\ana (R_1,R_2)\overset{\Adm}\longrightarrow\Corr_\adm (R_1,R_2)$, where the first map takes $\varphi\in\Hom (R_1,R_2)$ to the correspondence
\begin{equation}   \label{e:morphism as correspondence}
R_1\overset{\id}\longleftarrow R_1\overset{\varphi}\longrightarrow R_2. 
\end{equation}

\begin{lem}   \label{l:split butterflies}
(i) Let $\alpha\in\Corr_\adm (R_1,R_2)$ be an admissible correspondence 
\begin{equation}   \label{e:3correspondence in DGRings}
R_1\overset{f}\longleftarrow R_{12}\overset{g}\longrightarrow R_2
\end{equation}
Let $\adm^{-1}(\alpha )$ be the fiber of \eqref{e:Spl} over $\alpha$ (i.e., the set of isomorphism classes of pairs consisting of an element $\varphi\in\Hom (R_1,R_2)$ and an isomorphism $\adm (\varphi )\iso\alpha$). Then $\adm^{-1}(\alpha )$ canonically identifies with the set of splittings
\begin{equation} \label{e:Splittings}
\{ s:R_1\to R_{12}\,|\,  f\circ s=\id_{R_1} \},
\end{equation}
and after this identification  the map $\adm^{-1}(\alpha )\to\Hom (R_1,R_2)$ is given by $s\mapsto g\circ s$.

(ii) The  map from the set  \eqref{e:Splittings} to the set
\begin{equation} \label{e:splittings}
\{ \sigma:R_1^0\to R_{12}^0\,|\,  f^0\circ\sigma=\id_{R_1^0} \}
\end{equation}
given by $\sigma=s^0$ is bijective.
\end{lem}

The lemma implies that an admissible correspondence \eqref{e:3correspondence in DGRings} belongs to the essential image of \eqref{e:Spl} if and only if the homomorphism $f^0:R_{12}^0\to R_1^0$ admits a splitting. In terms of butterflies, this means that the NE-SW exact sequence in \eqref{e:butterfly} admits a splitting; in \S 4.5 of \cite{AN1} such butterflies are called \emph{splittable.}

\begin{proof}
Combining the definitions of $\adm$ and $\Adm$ with Lemma~\ref{l:Corr_adm is a groupoid}, we see that
an element of $\adm^{-1}(\alpha )$ is the same as an element $\varphi\in\Hom (R_1,R_2)$ plus a morphism from the correspondence
\eqref{e:morphism as correspondence} to the correspondence \eqref{e:3correspondence in DGRings}. Such a morphism is the same as a homomorphism $s:R_1\to R_{12}$ such that $g\circ s=\varphi$. This proves (i).

Statement (ii) easily follows from $f$ being a surjective quasi-isomorphism.
\end{proof}

\subsubsection{Explicit description of $\adm (\varphi )$}  \label{sss:adm (varphi) explicitly}
By Lemma~\ref{l:split butterflies}, the groupoid of admissible correspondences \eqref{e:3correspondence in DGRings} equipped with a splitting $\sigma :R_1^0\to R_{12}^0$ is a set, which identifies with $\Hom (R_1,R_2)$ as follows: given an admissible correspondences \eqref{e:3correspondence in DGRings} and a splitting $\sigma :R_1^0\to R_{12}^0$ one defines $\varphi\in\Hom (R_1,R_2)$ by $\varphi =g\circ s$, where $s:R_1\to R_{12}$ is the unique splitting of $f:R_{12}\to R_1$ with $s^0=\sigma$.

Let us describe the construction in the opposite direction in terms of butterflies. Given $\varphi\in\Hom (R_1,R_2)$, we have to construct a butterfly
\[
  \vcenter{
  \xymatrix@=1pc{
    R_1^{-1}\ar[dd]_d \ar[dr]^{h_1}  & &R_2^{-1} \ar[dl]_{h_2} \ar[dd]^d\\
    & R_{12}^0\ar[dl]_{f^0} \ar[dr]^{g^0} &  \\
    R_1^0 & & R_2^0
  }}
\]
equipped with a splitting $\sigma :R_1^0\to R_{12}^0$. One checks that the answer is as follows:
as an additive group,
\[
R_{12}^0=\{(x,y)\,|\, x\in R_1^0, y\in R_2^{-1}\},
\]
the multiplication operation  in $R_{12}^0$ is given by
\[
(x,y)\cdot (x',y')=(xx', \varphi^0 (x)y'+\varphi^0 (x')y+dy\cdot y'), 
\]
and the maps $h_1,h_2,f^0,g^0$ and $\sigma:R^0\to R_{12}^0$ are given by
\[
h_2(y)=(0,y), \quad f^0(x,y)=x, \quad \sigma (x)=(x,0),
\]
\[
g^0(x,y)=\varphi^0(x)+dy, \quad h_1(z)=(dz,-\varphi^{-1}(z)).
\]

\subsection{The $(2,1)$-category $\RGrpds_\AN$}     \label{ss:RGrpds_AN}
\subsubsection{Definition}    \label{sss:RGrpds_AN}
Following \cite{N0,AN1} and other works by Aldrovandi and Noohi, we define a $(2,1)$-category $\RGrpds_\AN$ as follows:

(a) its objects are those of $\DGRings^{0,-1}$;

(b) for $R_1,R_2\in\DGRings^{0,-1}$, the groupoid of 1-morphisms from $R_1$ to $R_2$ is $\Corr_\adm (R_1,R_2)$;

(c) the composition of 1-morphisms is the admissibilization of their composition as correspondences.

Good news: the composition of admissible correspondences is weakly admissible, so its admissibilization is as described in Lemma~\ref{l:weakly admissible}(iv). 
So the reader can easily describe composition of 1-morphisms in $\RGrpds_\AN$ using the language of butterflies from \S\ref{sss:butterflies} (on the other hand, the answer can be found in \cite[\S 10.1]{N0} or \cite[\S 5.1.1]{AN1}).

To check that $\RGrpds_\AN$ is indeed a 2-category, one has to prove the existence of the identity 1-morphisms. 
In fact, the identity endomorphism of the object of $\RGrpds_\AN$ corresponding to
$R\in\DGRings^{0,-1}$ is $\adm (\id_R)$, where $\adm :\Hom (R,R)\to\Corr_\adm (R,R)$ is as in \S\ref{sss:Spl}; this follows from \S\ref{sss:Adm of composition}. An explicit description of $\adm (\id_R)$ was given in \S\ref{sss:adm (varphi) explicitly}.

\subsubsection{Exercise} The 1-morphism given by an admissible correspondence $R_1\overset{f}\longleftarrow R_{12}\overset{g}\longrightarrow R_2$ is invertible if and only if $f$ is a quasi-isomorphism; in this case the inverse 1-morphism is the correspondence $R_2\overset{g}\longleftarrow R_{12}\overset{f}\longrightarrow R_1\,$.

\subsubsection{The functor $\DGRings^{0,-1}\to\RGrpds_\AN$}   \label{sss:DGRings to RGrpds_AN}
Recall that $\DGRings^{0,-1}$ and $\RGrpds_\AN$ have the same objects. We define a functor
\[
\DGRings^{0,-1}\to\RGrpds_\AN
\]
as follows: at the level of objects, it is the identity, and at the level of morphisms it is given by the map
 $\adm :\Hom (R_1,R_2)\to\Corr_\adm (R_1,R_2)$ from \S\ref{sss:Spl}. Compatibility with composition of morphisms follows from \S\ref{sss:Adm of composition}.

\subsection{The equivalence $\RGrpds_\AN\iso\RGrpds$}  \label{ss:RGrpds_AN to RGrpds}

\subsubsection{The 2-category $\RGrpds_\AN^\ana$}  \label{sss:RGrpds_AN^ana}
Define a $2$-category $\RGrpds_\AN^\ana$ as follows:

(a) its objects are those of $\DGRings^{0,-1}$;

(b) for $R_1,R_2\in\DGRings^{0,-1}$, the category of 1-morphisms from $R_1$ to $R_2$ is $\Corr_\ana (R_1,R_2)$;

(c) the composition of 1-morphisms is the usual composition of correspondences.

\subsubsection{The functor $\DGRings^{0,-1}\to\RGrpds$}  \label{sss:DGRings^{0,-1} to RGrpds}
Recall that $\RGrpds$ is the $(2,1)$-category of ring objects in the $(2,1)$-category of groupoids.
So we have a tautological functor $\RGrpds'_1\to\RGrpds$, 
where $\RGrpds'_1$ is the category of ring objects in the naive $1$-category of groupoids.
In \S\ref{ss:quasi-ideals}-\ref{ss:Q=RGrpds'} we constructed equivalences
$$\RGrpds'_1\iso\cQ\iso\DGRings^{0,-1}.$$
So we get a functor
\begin{equation}   \label{e:DGRings to RGrpds}
\DGRings^{0,-1}\to\RGrpds .
\end{equation}

\subsubsection{The functor $\RGrpds_\AN^\ana\to\RGrpds$}  
By \S\ref{sss:q-isomorphisms in Q}, the functor \eqref{e:DGRings to RGrpds} takes quasi-isomor\-phisms in $\DGRings^{0,-1}$ to equivalences between ring groupoids. 
So the functor \eqref{e:DGRings to RGrpds} canonically extends to a functor 
\begin{equation}   \label{e:RGrpds_AN^ana to RGrpds}
\RGrpds_\AN^\ana\to\RGrpds.
\end{equation}

\subsubsection{The functor $\RGrpds_\AN\to\RGrpds$}  \label{sss:RGrpds_AN to RGrpds}
We have a canonical functor 
\begin{equation}  \label{e:RGrpds_AN^ana to RGrpds_AN}
\RGrpds_\AN^\ana\to\RGrpds_\AN \, , 
\end{equation}
which acts as identity at the level of objects and as $\Adm :\Corr_\ana (R_1,R_2)\to\Corr_\adm (R_1,R_2)$
at the level of morphisms. By \S\ref{sss:Admissibilization as localization}, the functor \eqref{e:RGrpds_AN^ana to RGrpds_AN} has the following universal property:
for any $(2,1)$-category $\cC$, any functor $\RGrpds_\AN^\ana\to\cC$ uniquely factors through $\RGrpds_\AN\,$.
In particular, the functor \eqref{e:RGrpds_AN^ana to RGrpds} induces a functor
\begin{equation}   \label{e:RGrpds_AN to RGrpds}
\RGrpds_\AN\to\RGrpds.
\end{equation}

\begin{thm}  \label{t:main}
The functor \eqref{e:RGrpds_AN to RGrpds} is an equivalence.
\end{thm}

\subsubsection{Reducing Theorem~\ref{t:main} to Proposition~\ref{p:fiber has final object}}
Let $R_i\in\DGRings^{0,-1}$ and let $\sR_i\in~\RGrpds$ be the image of $R_i$ under the functor \eqref{e:DGRings to RGrpds}. The functor \eqref{e:RGrpds_AN^ana to RGrpds} induces functors
\begin{equation}   \label{e:Corr_ana to MMor}
\Corr_{\ana}(R_1,R_2)\to\MMor_{\RGrpds} (\sR_1,\sR_2),
\end{equation}
\begin{equation}   \label{e:Corr_adm to MMor}
\Corr_{\adm}(R_1,R_2)\to\MMor_{\RGrpds} (\sR_1,\sR_2),
\end{equation}
where $\MMor_{\RGrpds} (\sR_1,\sR_2)$ is the  groupoid of 1-morphisms $\sR_1\to\sR_2$.

By Proposition~\ref{p:essential surjectivity}, the functor \eqref{e:RGrpds_AN to RGrpds} is essentially surjective, so it remains to show that the functor \eqref{e:Corr_adm to MMor} is an equivalence. Since $\MMor_{\RGrpds} (\sR_1,\sR_2)$ is a groupoid, this follows from the next

\begin{prop}  \label{p:fiber has final object}
Let $\Phi\in\MMor_{\RGrpds} (\sR_1,\sR_2)$. Let $\Corr^\Phi_{\ana}(R_1,R_2)$ and $\Corr^\Phi_{\adm}(R_1,R_2)$ be the fibers of \eqref{e:Corr_ana to MMor} and
\eqref{e:Corr_adm to MMor} over $\Phi$. Then

(i) the category $\Corr^\Phi_{\ana}(R_1,R_2)$ has a final object;

(ii) the final object of $\Corr^\Phi_{\ana}(R_1,R_2)$ belongs to $\Corr^\Phi_{\adm}(R_1,R_2)$.
\end{prop}

The proof given below uses the fiber product in the $(2,1)$-category of ring groupoids, see~\S\ref{sss:fiber products}.

\begin{proof}
Statement (ii) follows from (i): if $X\in\Corr_{\ana}(R_1,R_2)$ is the image of a final object of $\Corr^\Phi_{\ana}(R_1,R_2)$ then the morphism $X\to\Adm (X)$ is an isomorphism, so $X$ is admissible.

To prove (i), we will use the equivalence
\begin{equation}   \label{e:DGRings=RGrpds'_3}
\DGRings^{0,-1}\iso\RGrpds'_3 \, ,
\end{equation}
from \S\ref{s:the 1-category}, where $\RGrpds'_3$ is as in \S\ref{sss:RGrpds'_3}. In particular, we think of $R_n\in\DGRings^{0,-1}$ as a pair 
$(\sR_n,\pi_n:R_n^0\epi\sR_n)\in \RGrpds'_3$. The equivalence \eqref{e:DGRings=RGrpds'_3} identifies $\Corr^\Phi_{\ana}(R_1,R_2)$ with the category of 2-commutative diagrams
\[
\xymatrix{
R_1^0 \ar[rd]_{\Phi\circ\pi_1}& R_{12}^0\ar@{->>}[l]\ar[r]&R_2^0\ar@{->>}[ld]^{\pi_2}  \\
& \sR_2& 
}
\]
So  $\Corr^\Phi_{\ana}(R_1,R_2)$ has a final object: it corresponds to the diagram
\begin{equation}    \label{e:2-categoical fiber product}
\xymatrix{
R_1^0 \ar[rd]_{\Phi\circ\pi_1}& R_1^0\times_{\sR_2}R_2^0\ar@{->>}[l]\ar[r]&R_2^0\ar@{->>}[ld]^{\pi_2} \\
& \sR_2& 
}
\end{equation}
(the map $R_1^0\times_{\sR_2}R_2^0\to R_1^0$ is surjective because $\pi_2:R_2^0\to\sR_2$ is essentially surjective).
\end{proof}

The above proof of Proposition~\ref{p:fiber has final object}(ii) was somewhat indirect (we used the functor $\Adm$ and the morphism $\Id\to\Adm$). In \S\ref{sss:direct proof} we will give a direct proof of this fact. But first we have to discuss the notion of admissible correspondence in a more abstract context.

\subsection{Admissible correspondences between categories}  \label{ss:admissibility for categories}
\subsubsection{Notation}  \label{sss:Cats' notation}
Let $\Cats$ be the 2-category of essentially small categories.  Let $\Cats'$ be the most naive 1-category of small categories (its morphisms are functors on the nose). 
Recall that if $\cC_1,\cC_2$ are categories we write $\Funct (\cC_1,\cC_2)$ for the category of functors $\cC_1\to\cC_2$.

Given $\cC_1,\cC_2\in\Cats'$, let $\Corr (\cC_1,\cC_2)$ be the category of correspondences from $\cC_1$ to $\cC_2$ in $\Cats'$ (see \S\ref{ss:correspondences}(ii)).
Thus objects of $\Corr (\cC_1,\cC_2)$ are diagrams
\begin{equation}   \label{e:3correspondence in Cats'}
\cC_1\overset{F}\longleftarrow \cC_{12}\overset{G}\longrightarrow \cC_2
\end{equation}
in $\Cats'$. 

\subsubsection{The graph of a functor}   \label{sss:2graph of functor}
By the \emph{graph} of a functor $\Phi:\cC_1\to\cC_2$ we mean the category of triples $(c_1,c_2,\psi)$, where $c_i\in\cC_i$ and $\psi$ is an isomorphism
$\Phi (c_1)\iso c_2$. We denote this category by $\Graph_\Phi$. The correspondence $\cC_1\leftarrow \Graph_\Phi\to\cC_2$ will be denoted by 
$\Graph (\Phi )$. Thus we get a functor
\begin{equation}  \label{e:2Graph}
\Graph :\Funct (\cC_1,\cC_2)\to\Corr (\cC_1,\cC_2), \, \quad \Phi\mapsto \Graph (\Phi ).
\end{equation}

\begin{lem}   \label{l:Graph is fully faithful}
The functor \eqref{e:2Graph} is fully faithful.  \qed
\end{lem}

\subsubsection{Admissible correspondences}   \label{sss:Admissible correspondences}
Let $\Corr_\adm (\cC_1,\cC_2)$ denote the essential image of \eqref{e:2Graph}. We say that a correspondence \eqref{e:3correspondence in Cats'} is \emph{admissible\footnote{Our terminology is not standard. In \cite{Mak} admissible correspondences are called \emph{saturated anafunctors.}}} if it belongs to
$\Corr_\adm (\cC_1,\cC_2)$. In this case $F:\cC_{12}\to\cC_1$ is a strictly surjective equivalence.

\begin{lem}   \label{l:2admissibility criterion}
Let
\begin{equation}   \label{e:4correspondence in Cats'}
\cC_1\overset{F}\longleftarrow \cC_{12}\overset{G}\longrightarrow \cC_2
\end{equation}
be a correspondence in $\Cats'$ such that $F$ is an equivalence. Then the following are equivalent:

(a) the correspondence \eqref{e:4correspondence in Cats'} is admissible;

(b) the canonical map 
\begin{equation}  \label{e:map required to be bijective}
\Ob\cC_{12}\to\Ob\cC_1\times_{\cC_2}\Ob\cC_2
\end{equation}
 is bijective; here 
\[
\Ob\cC_1\times_{\cC_2}\Ob\cC_2:=\{ (c_1,c_2,\psi )\,|\, c_i\in\Ob\cC_i, \;\; \psi: GF^{-1}(c_1)\iso c_2 \}
\]
is the fiber product in the 2-category $\Cats$, 
and the map \eqref{e:map required to be bijective} takes $c\in\cC_{12}$ to  the triple $(c_1,c_2,\psi )$, where $c_1=F(c)$, $c_2=G(c)$, and 
$\psi :GF^{-1}F(c)\iso G(c)$ comes from the isomor\-phism~$F^{-1}F(c)\iso c$;

c) the functor
\[
H:\cC_{12}\to\cC_1\times\cC_2, \quad H:=(F,G)
\]
has the following property: for every $c\in\cC_{12}$ every $(\cC_1\times\cC_2)$-isomorphism with source $H(c)$ has one and only lift to a $\cC_{12}$-isomorphism with source $c$.   \qed
\end{lem}

\subsection{Comparing the two notions of admissibility}  \label{ss:two notions of admissibility}
Let 
\begin{equation}   \label{e:correspondence in RGrpds'}
\sR_1\leftarrow \sR_{12}\rightarrow \sR_2
\end{equation}
be a correspondence in $\RGrpds'_1$ (i.e., in the 1-category of ring groupoids). It is said to be admissible if its image under the equivalence $\RGrpds'_1\iso\DGRings^{0,-1}$ is an admissible correspondence in $\DGRings^{0,-1}$ (as defined in \S\ref{sss:Three classes}).

On the other hand, applying to \eqref{e:correspondence in RGrpds'} the functor $\RGrpds'_1\to\Grpds'_1\subset\Cats'$ from \S\ref{ss:RGrpds'}, we get a correspondence in $\Cats'$. For such correspondences we have the notion of admissibility from \S\ref{sss:Admissible correspondences}.

\begin{lem}    \label{l:two notions of admissibility}
A correspondence $\sR_1\leftarrow \sR_{12}\rightarrow \sR_2$ in $\RGrpds'_1$ is admissible if and only if its image in $\Cats'$ is an admissible correspondence  in the sense of \S\ref{sss:Admissible correspondences}.
\end{lem}

\begin{proof}
Use the equivalence (a)$\Leftrightarrow$(c) from Lemma~\ref{l:2admissibility criterion} (combined with \S\ref{sss:description of Cone (d)}-\ref{sss:q-isomorphisms in Q}). 
\end{proof}

\subsubsection{A direct proof of Proposition~\ref{p:fiber has final object}(ii)}   \label{sss:direct proof} 
We have to show that the correspondence in $\DGRings^{0,-1}$ (or equivalently, in $\RGrpds'_1$) corresponding to diagram \eqref{e:2-categoical fiber product} is admissible. This follows from Lemma~\ref{l:two notions of admissibility} and the equivalence (a)$\Leftrightarrow$(c) from Lemma~\ref{l:2admissibility criterion}. \qed

\subsubsection{On the functor $\MMor_{\RGrpds} (\sR_1,\sR_2)\iso\Corr_{\adm}(R_1,R_2)$}
Let $\sR_1,\sR\in~\RGrpds$ correspond to $R_1,R_2\in\DGRings^{0,-1}$. In the proof of Proposition~\ref{p:fiber has final object} we constructed the functor $\MMor_{\RGrpds} (\sR_1,\sR_2)\to\Corr_{\adm}(R_1,R_2)$ inverse to  \eqref{e:Corr_adm to MMor}. The same functor can be described as follows.

First of all, $\Corr_{\adm}(R_1,R_2)=\Corr_{\adm}(\sR'_1,\sR'_2)$, where $\sR'_i\in\RGrpds'_1$ is the image of $R_i$ under the equivalence
$\DGRings^{0,-1}\iso\RGrpds'_1$. To construct the functor $$\MMor_{\RGrpds} (\sR_1,\sR_2)\to\Corr_{\adm}(\sR'_1,\sR'_2),$$ recall that
by definition, $\sR'_i$ is a functor $F_i:\Pol^{\op}\to\Grpds'_1$. A 1-morphism $\sR_1\to\sR_2$ defines for each $A\in\Pol^{\op}$ a functor from the groupoid  $F_1(A)$ to the groupoid $F_2(A)$. Its graph (in the sense of \S\ref{sss:2graph of functor}) is a correspondence $F_1(A)\leftarrow F_{12}(A)\to F_2(A)$. Note that $F_{12}(A)$ is defined
up to unique isomorphism, i.e., as an object of $\Grpds'_1$. The assignment $A\mapsto F_{12}(A)$ is  a functor $$\Pol^{\op}\to\Grpds'_1$$ commuting with products, i.e., an object of $\RGrpds'_1$. Thus we get an object of $\Corr (\sR'_1,\sR'_2)$. By Lemma~\ref{l:two notions of admissibility}, it is in $\Corr_{\adm}(\sR'_1,\sR'_2)$.

\section{Anafunctors}     \label{s:anafunctors}
We keep the notation of \S\ref{sss:Cats' notation}-\ref{sss:2graph of functor}; in particular, $\Cats'$ stands for the most naive $1$-ca\-te\-gory of categories (its morphisms are functors on the nose). In this section we continue the discussion of correspondences in $\Cats'$, which began in \S\ref{ss:admissibility for categories}; we will see that the picture is quite parallel to that of \S\ref{ss: Corr in DGRings}-\ref{ss:Admissibilization}. We mostly follow M.~Makkai \cite{Mak}.

\subsection{Four classes of correspondences in $\Cats'$}

Given $\cC_1,\cC_2\in\Cats'$, we defined in \S\ref{sss:Admissible correspondences} a strictly full subcategory $\Corr_{\adm}(\cC_1,\cC_2)\subset\Corr(\cC_1,\cC_2)$. Now we are going to define strictly full subcategories
$$\Corr_{\wadm}(\cC_1,\cC_2)\subset\Corr_{\ana}(\cC_1,\cC_2)\subset\Corr_{\eq}(\cC_1,\cC_2)\subset\Corr(\cC_1,\cC_2) $$
containing $\Corr_{\adm}(\cC_1,\cC_2)$.

\subsubsection{$\Corr_{\eq}(\cC_1,\cC_2)$ and $\Corr_{\ana}(\cC_1,\cC_2)$}   \label{sss:Anafunctors}
Define strictly full subcategories
\[
\Corr_{\ana} (\cC_1,\cC_2)\subset\Corr_{\eq} (\cC_1,\cC_2)\subset\Corr (\cC_1,\cC_2)
\]
as follows:  a correspondence 
\begin{equation}   \label{e:correspondence in Cats'}
\cC_1\overset{F}\longleftarrow \cC_{12}\overset{G}\longrightarrow \cC_2
\end{equation}
 is in $\Corr_{\eq}(\cC_1,\cC_2)$ (resp.~in $\Corr_{\ana}(\cC_1,\cC_2)$) if and only if 
$F: \cC_{12}\to\cC_1$ is an equivalence (resp.~a strictly surjective equivalence). 

Objects of $\Corr_{\ana} (\cC_1,\cC_2)$ are called \emph{anafunctors}; this terminology goes back to  M.~Makkai's work \cite{Mak}. Warning: Makkai's notion of \emph{morphism} of anafunctors is \emph{different} from ours (his category of anafunctors is equivalent to $\Funct (\cC_1,\cC_2)$).

\subsubsection{The diagram $\Corr_{\eq} (\cC_1,\cC_2)\rightleftarrows\Funct (\cC_1,\cC_2)$}  \label{sss:the adjoint pair}
(i) In \S\ref{sss:2graph of functor} we defined a fully faithful functor
\begin{equation}  \label{e:Graph}
\Graph :\Funct (\cC_1,\cC_2)\to\Corr_{\ana} (\cC_1,\cC_2)\subset\Corr_{\eq} (\cC_1,\cC_2). 
\end{equation}

(ii) On the other hand, we have the following functor $\Corr_{\eq} (\cC_1,\cC_2)\to\Funct (\cC_1,\cC_2)$:
 to a diagram \eqref{e:correspondence in Cats'} such that $F$ is an equivalence we associate $G\circ F^{-1}\in\Funct (\cC_1,\cC_2)$. Note that $F^{-1}$ and $G\circ F^{-1}$ are defined up to unique isomorphism.

\begin{lem}    \label{l:adjunction}
The functor $\Corr_{\eq} (\cC_1,\cC_2)\to\Funct (\cC_1,\cC_2)$ from \S\ref{sss:the adjoint pair}(ii) is left adjoint to the functor \eqref{e:Graph}.
 The unit of the adjunction is given by the functor 
\begin{equation}   \label{e:2admissibilization}
\cC_{12}\to\Graph_{GF^{-1}}
\end{equation}
 that takes $c\in\cC_{12}$ to $(c_1,c_2,\psi )$, where $c_1=F(c)$, $c_2=G(c)$, and 
$\psi :GF^{-1}F(c)\iso G(c)$ comes from the isomorphism $F^{-1}F(c)\iso c$. \qed
\end{lem}

\subsubsection{Admissible and weakly admissible correspondences}
In \S\ref{e:Graph} we defined the category $\Corr_{\adm} (\cC_1,\cC_2)$ of admissible correspondences to be the essential image of the functor \eqref{e:Graph}.
By Lemma~\ref{l:adjunction}, a correspondence \eqref{e:correspondence in Cats'} is admissible if and only if the functor \eqref{e:2admissibilization} is an isomorphism.

We say that a correspondence \eqref{e:correspondence in Cats'} is \emph{weakly admissible} if the functor \eqref{e:2admissibilization} is strictly surjective. Note that in our situation \emph{essential} surjectivity of \eqref{e:2admissibilization} is automatic; in fact, the functor \eqref{e:2admissibilization} is automatically an equivalence.

Let $\Corr_{\wadm} (\cC_1,\cC_2)\subset\Corr_{\eq} (\cC_1,\cC_2)$ be the full subcategory of weakly admissible
correspondences. 
One has
\[
\Corr_{\wadm} (\cC_1,\cC_2)\subset\Corr_{\ana} (\cC_1,\cC_2)
\]
because strict surjectivity of \eqref{e:2admissibilization} implies strict surjectivity of $F:\cC_{12}\to\cC_1$.

In Lemmas~\ref{l:2admissibility criterion} we gave two criteria for admissibility. 
This lemma remains valid if one replaces ``admissible'' by ``weakly admissible'' , replaces ``bijective'' by ``surjective'' in (b) and removes the words ``only one'' from (c).

\subsubsection{Admissibilization}
The functor $\Corr_{\eq} (\cC_1,\cC_2)\to\Corr_{\adm} (\cC_1,\cC_2)$ obtained by composing the two functors from \S\ref{sss:the adjoint pair}
%\S\ref{sss:GF^{-1}} and \S\ref{sss:graph of functor} 
is called \emph{admissibilzation}\footnote{The name used by Makkai \cite{Mak} is \emph{saturation}. His name for ``admissible correspondence'' is ``saturated anafunctor''.}.

If a correspondence \eqref{e:correspondence in Cats'} is weakly admissible then its admissibilization
is the correspondence $\cC_1\leftarrow \tilde\cC_{12}\to\cC_2$ obtained by setting $\Ob\tilde\cC_{12}:=(\Ob\cC_{12}/)R$, where $R$ is the following equivalence relation: $c\sim c'$ if $F(c)=F(c')$, $G(c)=G(c')$, and the unique isomorphism $\alpha :c\iso c'$ with $F(\alpha )=\id$ satisfies $G(\alpha )=\id$. (We do not have to worry about the morphisms of 
$\tilde\cC_{12}$ and their composition: they come from $\cC_1$).

\subsection{Composition of functors and correspondences}
As before, we follow \cite{Mak}.

\subsubsection{Composition of correspondences}
According to the general definition from \S\ref{ss:correspondences}(iii), correspondences in $\Cats'$ are composed as follows: the composition of correspondences 
$$\cC_1\leftarrow \cC_{12}\to \cC_2 \mbox{ and } \cC_2\leftarrow \cC_{23}\to \cC_3$$ is  the correspondence $\cC_1\leftarrow \cC_{12}\times_{\cC_2} \cC_{23}\to \cC_3$, where 
$\cC_{12}\times_{\cC_2} \cC_{23}$ is the fiber product in $\Cats'$ (rather than in $\Cats$). In general, the fiber product in $\Cats'$ is ``not really good''. But it is good enough if the correspondences in question are anafunctors. 

One checks that the composition of anafunctors is an anafunctor, and the composition of weakly admissible correspondences is weakly admissible.
Moreover, the construction of 
\S\ref{sss:the adjoint pair}(ii) takes composition of anafunctors to composition of functors.

Thus we get the 2-category whose objects are categories, whose 1-morphisms are anafunctors, and whose 2-morphisms are 2-morphisms between correspondences; we also get a functor from this 2-category to $\Cats$.

Let us note that the composition of admissible correspondences is usually \emph{not admissible}. In other words, given functors 
$\cC_1\overset{\Phi}\longrightarrow\cC_2\overset{\Psi}\longrightarrow\cC_3$, we have a canonical surjective equivalence $\Graph_\Phi\times_{\cC_2}\Graph_\Psi\to\Graph_{\Psi\circ\Phi}$, but it is \emph{not an isomorphism}. Indeed, an object of $\Graph_{\Psi\circ\Phi}$ is given by objects $c_1\in\cC_1, c_3\in\cC_3$, and an isomorphism $\Psi (\Phi (c_1))\iso c_3$; on the other hand, to specify an object of $\Graph_\Phi\times_{\cC_2}\Graph_\Psi$, one needs, in addition, an object $c_2\in\cC_2$ and an isomorphism $\Phi (c_1)\iso c_2$.

\subsubsection{$\Cats_\AN$ as a ``model'' for $\Cats$}
Thus we see that the 2-category $\Cats$ is equivalent to  the $2$-category $\Cats_\AN$ defined as follows:

(a) the objects of $\Cats_\AN$ are categories;

(b) the category of 1-morphisms from an object $\cC_1$ to and object $\cC_2$ is $\Corr_\adm (\cC_1,\cC_2)$;

(c) the composition of 1-morphisms is the admissibilization of their composition as correspondences in $\Cats'$.

The above definition of  $\Cats_\AN$ is parallel to the definition of $\RGrpds_\AN$ from \S\ref{sss:RGrpds_AN}.

\bibliographystyle{alpha}

\end{document}